\documentclass[11pt]{article}
\usepackage{amsthm,amsfonts,amsmath,amscd,mathabx,hyperref,color}
\hypersetup{
     colorlinks   = true,
     citecolor    = blue
}
\usepackage{cite}
\usepackage[toc, page]{appendix}

\swapnumbers
\theoremstyle{plain}

\newtheorem{thm}{THEOREM}[section]
\newtheorem{lem}[thm]{LEMMA}
\newtheorem{cor}[thm]{COROLLARY}

\theoremstyle{definition}

\theoremstyle{definition}
\newtheorem{rem}[thm]{Remark}
\newcommand{\upchi}{\raise1pt\hbox{$\chi$}}
\newcommand{\R}{{\mathord{\mathbb R}}}

\newcommand{\Z}{{\mathord{\mathbb Z}}}
\newcommand{\N}{{\mathord{\mathbb N}}}

\numberwithin{equation}{section}
\def\keywords{\vspace{.5em}
{\textit{Keywords}:\,\relax%
}}

\def\MSC{\vspace{.5em}
{2020 \textit{Mathematics Subject Classification}:\,\relax%
}}

\begin{document}

\title{\bf One-dimensional discrete\\ Hardy and Rellich inequalities on integers}

\author{\vspace{5pt} By Shubham Gupta\\ \date{}
\vspace{5pt}\small{Department of Mathematics, Imperial College London}\\[-6pt]
\small{London SW7 2AZ, UK}\\
\small{Email address: s.gupta19@imperial.ac.uk}
}

\maketitle 
\hypersetup{linkcolor=blue}

\textbf{Abstract}:   
In this paper, we consider a weighted version of one-dimensional discrete Hardy inequalities with power weights of the form $n^\alpha$. We prove the inequality when $\alpha$ is an even natural number with the sharp constant and remainder terms. We also find explicit constants in standard and weighted Rellich inequalities(with weights $n^\alpha$) which are asymptotically sharp as $\alpha \rightarrow \infty$. As a by-product of this work we derive a combinatorial identity using purely analytic methods, which suggests a plausible correlation between combinatorial and functional identities.\\\\
\\
\MSC{Primary: 39B62; Secondary: 05A19.}

\section{Introduction}\label{sec1}
The classical \textbf{Hardy inequality} on the positive half-line reads as
\begin{equation}\label{1.1}
    \int_0^\infty |u'(x)|^2 dx \geq 1/4 \int_0^\infty \frac{|u(x)|^2}{x^2} dx,
\end{equation}
for $u \in C_0^\infty (0, \infty)$, the space of smooth and compactly supported functions. This inequality first appeared in Hardy's proof of Hilbert's theorem \cite{Hardy1920} and then in the book \cite{HLP1952}. It was later extended to higher order derivatives by Birman \cite{birman}: Let $k \in \N$ and $u \in C_0^\infty(0, \infty)$, then
\begin{equation}\label{1.2}
    \int_0^\infty |u^{(k)}(x)|^2 dx \geq \frac{2^{2k}}{[(2k-1)!!]^2} \int_0^\infty \frac{|u(x)|^2}{x^{2k}} dx,
\end{equation}
where $u^{(k)}$ denotes the $k^{th}$ derivative of $u$ and $(2k-1)!! = (2k-1)(2k-3)(2k-5) \dots 3 \cdot 1$. Inequality \eqref{1.2} for $k=2$ is referred to as \textbf{Rellich inequality}. Note that constants in \eqref{1.1} and \eqref{1.2} are \emph{sharp}, that is, these inequalities fail to hold true for a strictly bigger constant.  \\

The main goal of this paper is to study a discrete analogue of \eqref{1.1} and \eqref{1.2} as well as their weighted versions on integers. A well known discrete variant of \eqref{1.1} states: Let $\N_0$ denotes the set of non-negative integers and $u: \N_0 \rightarrow \R$ be a finitely supported function with $u(0)=0$. Let $Du(n) := u(n)-u(n-1)$ denote the \emph{first order difference operator} on $\N_0$. Then \begin{equation}\label{1.3}
    \sum_{n=1}^ \infty |Du(n)|^2 \geq \frac{1}{4} \sum_{n=1}^\infty \frac{|u(n)|^2}{n^2}.
\end{equation}
The constant in \eqref{1.3} is sharp. This inequality was developed alongside the integral inequality \eqref{1.1} during the period 1906-1928: \cite{KMP2006} contains many stories and contributions of other mathematicians such as E. Landau, G. Polya, I. Schur and M. Riesz in the development of Hardy inequality. We would also like to mention some recent proofs of Hardy inequality \eqref{1.3} \cite{fischer2019improved, KPP18, krejcirik2021spectral, krejcirik2021sharp, lef2020} as well as \cite{berchio2021poincare, bui2020application, braverman1994discrete, gupta2022discrete, kapitanski2016continuous, keller2018optimal, keller2021hardy, kostenko2021heat, miclo1999example, liu2012some}, where various variants of \eqref{1.3} have been studied and applied: extensions of \eqref{1.3} to higher dimensional integer lattice, on combinatorial trees, general weighted graphs, etc. \\

In this paper, we are concerned with an extension of inequality \eqref{1.3} in two directions. First, we consider a weighted version of \eqref{1.3} with power weights $n^\alpha$: 
\begin{equation}\label{1.4}
    \sum_{n=1}^\infty |Du(n)|^2 n^\alpha \geq c \sum_{n=1}^\infty \frac{|u(n)|^2}{n^2} n^\alpha,
\end{equation}
for some positive constant $c$. We prove inequality \eqref{1.4} with the sharp constant and furthermore improve it by adding lower order remainder terms in the RHS. This is done when $\alpha$ is a non-negative even integer. This problem has been studied previously: in \cite{lefevre:hal-02528265}, \eqref{1.4} was proved when $ \alpha \in (0,1)$ and recently it was extended to $\alpha > 5$ in \cite{gupta2022discrete}. In this paper, we provide a new method to prove these inequalities, which extends and improves previously known results. \\

Secondly, we consider the higher order versions of inequality \eqref{1.3}, in which the ``discrete derivative'' on the LHS of \eqref{1.3} is replaced by higher order operators. In other words, we prove a discrete analogue of inequalities \eqref{1.2}. In particular, we find a constant $c(k)$ in the following Rellich inequality for finitely supported functions on non-negative integers:
\begin{equation}\label{1.5}
    \sum_{n=0}^\infty |\Delta u(n)|^2 \geq c \sum_{n=1}^\infty \frac{|u(n)|^2}{n^4},
\end{equation}
where the \emph{second order difference operator} on $\N_0$, called $\textbf{Laplacian}$ is given by
\begin{align*}
    \Delta u(n):= 
\begin{cases}
    2u(n)-u(n-1)-u(n+1), \hspace{9pt} \text{if} \hspace{5pt} n \in \N \\
    u(0) - u(1), \hspace{90pt} \text{if} \hspace{5pt} n=0
\end{cases}    
\end{align*}
This is a well known discrete analogue of second order derivative. Inequality \eqref{1.5} has been considered in the past in a general setting of graphs \cite{keller2021hardy,keller2021optimal}. In these papers, authors developed a general theory to tackle problems of the kind \eqref{1.5}, however; one cannot deduce the Rellich inequality \eqref{1.5} from their general theory. To the author's best knowledge, this is the first time an explicit constant has been computed in the discrete Rellich inequality \eqref{1.5}. We also prove inequality \eqref{1.5} with weights $n^{2k}$, for positive integers $k$. The constant obtained is asymptotically sharp as $k \rightarrow \infty$. \newpage

As a side product, we discovered a surprising connection between functional and combinatorial identities. Using purely analytic methods, we managed to prove a non-trivial combinatorial identity, whose appearance in the context of discrete Hardy-type inequalities seems mysterious. This connection will be explained in Sections \ref{sec3} and \ref{sec6}. We hope that the analytic method presented here might lead to the discovery of new combinatorial identities. \\

The paper is structured as follows: In Section \ref{sec2}, we state the main results of the paper. In Section \ref{sec3}, we prove auxiliary results using which we prove our main results in Sections \ref{sec4} and \ref{sec5}. In Section \ref{sec6}, we prove a combinatorial identity using lemmas proved in Section \ref{sec3}. Finally we conclude the paper with an appendix \ref{appendix:A}.
\begin{rem}
For the convenience of reader we would recommend that reader should read Sections \ref{sec4} and \ref{sec5} before reading Section \ref{sec3} to get a better understanding of ideas involved and origin of the lemmas proved in Section \ref{sec3}.   
\end{rem}

\section{Main Results}\label{sec2}
\subsection{Hardy inequalities}

\begin{thm}[Improved weighted Hardy inequalities]\label{thm2.1}
Let $u \in C_c(\mathbb{Z})$, the space of finitely supported functions, and also assume $u(0)=0$. Then for $k \in \N$, we have
\begin{equation}\label{2.1}
    \sum_{n \in \mathbb{Z}}|u(n)-u(n-1)|^2 \Big(n-\frac{1}{2}\Big)^{2k} \geq \sum_{i=1}^k \gamma_i^k \sum_{n \in \mathbb{Z}} |u(n)|^2 n^{2k-2i}+2^{-2k-2}\sum_{n \in \mathbb{Z}\setminus\{0\}} \frac{|u(n)|^2}{n^2},
\end{equation}
where the non-negative constants $\gamma_i^k$ are given by
\begin{equation}\label{2.2}
    2^{2i}\gamma_i^k := 2{2k \choose 2i} - 2{k \choose i} + {k \choose i-1}.
\end{equation}
Here $\Gamma(x)$ denotes the Gamma function and ${x \choose y} := \frac{\Gamma(x+1)}{\Gamma(x-y+1) \Gamma(y+1)}$ denotes the binomial coefficient.
\end{thm}
Dropping the remainder terms in inequality \eqref{2.1} gives the following weighted Hardy inequalities: \\
\begin{cor}[Weighted Hardy inequalities]\label{cor2.2}
Let $u \in C_c(\mathbb{Z})$ and $u(0)=0$. Then for $k \in \N $, we have
\begin{equation}\label{2.3}
    \sum_{n \in \mathbb{Z}}|u(n)-u(n-1)|^2 \Big(n-\frac{1}{2}\Big)^{2k} \geq \frac{(2k-1)^2}{4} \sum_{n \in \mathbb{Z}} |u(n)|^2 n^{2k-2}.
\end{equation}
Moreover, the constant $(2k-1)^2/4$ is sharp.
\end{cor}

We would like to mention that inequality \eqref{2.3} was proved in paper \cite{liu2012some} with the weight $(n-1/2)^\alpha$ and $\alpha \in (0,1)$. Note that the above inequalities reduce to corresponding Hardy inequalities on non-negative integers $\mathbb{N}_0$, when we restrict ourselves to functions $u$ taking value zero on negative integers. \newpage

Using the method used in the proofs of above Hardy inequalities, we also managed to prove higher-order versions of the Hardy inequality, in-particular we prove a discrete Rellich inequality which has been missing from the current literature. 

\subsection{Higher Order Hardy inequalities}
\begin{thm}[Higher order Hardy inequalities]\label{thm2.3}
Let $m\in \mathbb{N}$. Then we have 
\begin{equation}\label{2.4}
    \sum_{n=0}^\infty |\Delta^m u(n)|^2 \geq \frac{1}{2^{4m}} \prod_{i=0}^{2m-1} (8m-3-4i) \sum_{n=1}^\infty \frac{|u(n)|^2}{n^{4m}},
\end{equation}
for all $u \in C_c(\mathbb{N}_0)$ with $u(i)=0$ for $0\leq i \leq 2m-1$, \\
and 
\begin{equation}\label{2.5}
    \sum_{n=1}^\infty |D(\Delta^m u)(n)|^2 \geq \frac{1}{2^{4m+2}} \prod_{i=0}^{2m} (8m+1-4i) \sum_{n=1}^\infty \frac{|u(n)|^2}{n^{4m+2}},
\end{equation}
for all $u \in C_c(\mathbb{N}_0)$ with $u(i)=0$ for  $0\leq i \leq 2m$. Here $Du$ and $\Delta u$ denotes the first and second order difference operators on $\N_0$ respectively.
\end{thm}
Theorem \ref{thm2.3} is a discrete analogue of inequalities of Birman \eqref{1.2}. Inequality \eqref{2.4} for $m=1$ gives Rellich inequality:\\

\begin{cor}[Rellich inequality]\label{cor2.4}
Let $u \in C_c(\mathbb{N}_0)$ and $u(0)=u(1)=0$. Then we have 
\begin{equation}\label{2.6}
    \sum_{n=0}^\infty |\Delta u(n)|^2 \geq \frac{5}{16} \sum_{n=1}^\infty \frac{|u(n)|^2}{n^4}.
\end{equation}
\end{cor}
\begin{rem}
It is worthwhile to notice that in Theorem \ref{thm2.3} the number of zero conditions on the function $u$ equals the order of the operator. Whether the number of zero conditions are optimal or not is not clear to us. Furthermore, we don't believe the constants obtained in Theorem \ref{thm2.3} are sharp. There seems to be a lot of room for the improvement in the constants, though it is not clear how to get better explicit bounds.   
\end{rem}
Finally, we obtain explicit constants in weighted versions of higher order Hardy Inequalities. For functions $u : \Z \rightarrow \R$ we define the first and second order difference operators on $\Z$ analogously: $Du(n) := u(n)-u(n-1)$ and $\Delta u(n) := 2u(n)-u(n-1)-u(n+1)$.\\

\begin{thm}[Power weight higher order Hardy inequalities]\label{thm2.6}
Let $m \geq 1$ and $u \in C_c(\mathbb{Z})$ with $u(0)=0$. Let $Du$ and $\Delta u$ denote the first and second order difference operators on $\Z$. Then 
\begin{equation}\label{2.7}
    \sum_{n \in \mathbb{Z}} |\Delta^m u(n)|^2 n^{2k} \geq \prod_{i=0}^{m-1} C(k-2i) \sum_{n \in \mathbb{Z}} |u(n)|^2 n^{2k-4m}
\end{equation}
for $k \geq 2m$ and 
\begin{equation}\label{2.8}
    \sum_{n \in \mathbb{Z}}|D(\Delta^m u)(n)|^2  \Big(n-\frac{1}{2}\Big)^{2k} \geq \frac{(2k-1)^2}{4} \prod_{i=0}^{m-1} C(k-1-2i) \sum_{n \in \mathbb{Z}}|u(n)|^2 n^{2k-4m-2}
\end{equation}
for $k \geq 2m+1$,\\
where $C(k)$ is given by
\begin{equation}\label{2.9}
    C(k):= k(k-1)(k-3/2)^2.
\end{equation}
\end{thm}

\begin{rem}
By taking $n^\beta$ as test functions in the inequalities \eqref{2.7} and \eqref{2.8} it can be easily seen that the sharp constants in these inequalities are of the order $O(k^{4m})$ and $O(k^{4m+2})$ respectively. Therefore constants obtained in Theorem \ref{thm2.6} are asymptotically sharp as $k \rightarrow \infty$.
\end{rem}

\section{Some Auxiliary Results}\label{sec3}

\begin{lem}\label{lem3.1}
Let $u \in C^\infty([-\pi, \pi])$. Furthermore, assume that derivatives of $u$ satisfy $d^k u(-\pi) = d^ku(\pi)$ for all $k \in \mathbb{N}_0$. For every $k \in \mathbb{N}$ we have
\begin{equation}\label{3.1}
    \int_{-\pi}^{\pi} |d^k (u\sin(x/2))|^2 dx = \sum_{i=0}^k \alpha_i^k \int_{-\pi}^{\pi} |d^i u|^2 dx  + \sum_{i=0}^{k} \beta_i^k \int_{-\pi}^{\pi}|d^i u|^2\sin^2(x/2) dx,
\end{equation}
where 
\begin{equation}\label{3.2}
    2^{2(k-i)}\alpha_i^k := \frac{1}{2} {2k \choose 2i} - \frac{1}{2}(-1)^{k-i} {k \choose i}^2 - \frac{1}{2}(-1)^{k-i} \xi_i^k,
\end{equation}

\begin{equation}\label{3.3}
    2^{2(k-i)}\beta_i^k := (-1)^{k-i} \xi_i^k + (-1)^{k-i}{k \choose i}^2,
\end{equation}
and 

\begin{equation}\label{3.4}
    \xi_i^k := \sum_{\substack{0 \leq m \leq \text{min}\{i,k-i\} \\ 1 \leq n \leq k-i}} (-1)^n 2^{n-m} {k+1 \choose i-m} {k \choose i+n}{n-1 \choose m}.
\end{equation}

\end{lem}

\begin{proof}
Using the Leibniz product rule for the derivative we get
\begin{align*}
    |d^k (u(x)\sin(x/2))|^2 &= |\sum_{i=0}^k {k \choose i}d^iu(x)d^{k-i}\sin(x/2)|^2\\
    &= \sum_{i=0}^k {k \choose i}^2|d^iu(x)|^2|d^{k-i}\sin(x/2)|^2 \\
    & + 2\text{Re} \sum_{0 \leq i<j \leq k} {k \choose i}{k \choose j}d^i u(x) \overline{d^ju(x)} d^{k-i}\sin(x/2) d^{k-j}\sin(x/2).
\end{align*}

Integrating both sides, we obtain
\begin{equation}\label{3.5}
    \begin{split}
        \int_{-\pi}^{\pi} |d^k (u(x)\sin(x/2))|^2 &= \sum_{i=0}^k {k \choose i}^2 \int_{-\pi}^{\pi}     |d^iu(x)|^2|d^{k-i}\sin(x/2)|^2\\
        &+2\text{Re} \sum_{0 \leq i<j \leq k} {k \choose i}{k \choose j} \int_{-\pi}^{\pi} d^i u(x) \overline{d^ju(x)} d^{k-i}\sin(x/2) d^{k-j}\sin(x/2).
    \end{split}
\end{equation}

Let $0 \leq i<j$ and $I(i,j):=$Re $\int_{-\pi}^{\pi} d^i u(x) \overline{d^ju(x)} d^{k-i}\sin(x/2) d^{k-j}\sin(x/2)$. Applying integration by parts iteratively, we get
\begin{equation}\label{3.6}
    I(i,j)= \text{Re} \int_{-\pi}^{\pi} d^i u(x) \overline{d^ju(x)} d^{k-i}\sin(x/2) d^{k-j}\sin(x/2) = \sum_{\sigma = i}^{\lfloor\frac{i+j}{2}\rfloor} \int_{-\pi}^{\pi} C_{\sigma}^{i,j}(x) |d^\sigma u|^2,
\end{equation}
where $C_{\sigma}^{i,j}$ is given by
\begin{align*}
    C_{\sigma}^{i,j}(x) = {j-\sigma -1 \choose \sigma -i-1}(-1)^{j-\sigma}d^{i+j-2\sigma}w_{ij}(x)  + \frac{1}{2}{j-\sigma-1 \choose \sigma-i}(-1)^{j-\sigma}d^{i+j-2\sigma}w_{ij}(x),
\end{align*}
and $w_{ij}(x):= d^{k-i}\sin(x/2) d^{k-j}\sin(x/2)$. \footnote{See Appendix \ref{appendix:A} for a proof of identity \eqref{3.6}.}\\

Using \eqref{3.6} in \eqref{3.5}, we see that 
\begin{equation}\label{3.7}
    \int_{-\pi}^{\pi} |d^k(u(x)\sin(x/2))|^2 = \sum_{i=0}^{k} \int_{-\pi}^{\pi} D_i(x)|d^i u|^2,
\end{equation}
since the derivatives which appear in the expression of $I(i,j)$ are of order between $i $ and $\lfloor\frac{i+j}{2}\rfloor$. Observing that the terms which contributes to $D_i$ are of the form $I(i-m,i+n)$ with the condition $m\leq n$, we get the following expression for $D_i(x)$:
\begin{align*}
    D_i(x) = 2\sum_{\substack{0 \leq m \leq \text{min}\{i,k-i\}\\ m \leq n \leq k-i}} {k \choose i-m} {k \choose i+n}C_i^{i-m,i+n}(x),
\end{align*}
where $C_i^{i,i}(x):= \frac{1}{2}|d^{k-i}\sin(x/2)|^2$.\\

It can be checked that for non-negative integers $l$, $d^l w_{ij}(x) \in \{\sin^2(x/2), \cos^2(x/2), \cos x, \sin x \}$ (with some multiplicative constant). Thus $D_i(x)$ is a linear combination of $\sin^2(x/2), \cos^2(x/2), \newline \cos x$ and $\sin x$. Namely, we have
\begin{align*}
    D_{i}(x) = C_1^i \sin^2(x/2) + C_2^i \cos^2(x/2) + C_3^i \cos x + C_4^i \sin x.
\end{align*}
Note that $\sin^2(x/2)$ can appear in the expression of $D_i$ iff $w_{i-m,i+n}$ is a multiple of $\sin^2(x/2)$ and $m=n$. Further, observing that $w_{i-m,i+m}$ is a multiple of $\sin^2(x/2)$ iff $k-i+m$ is even, we get 
\begin{equation}\label{3.8}
    C_1^i = 2 \sum_{\substack{ 1 \leq m \leq \text{min}\{i,k-i\} \\ k-i+m \hspace{2pt}\text{is even}}} 2^{-2(k-i)} {k \choose i-m}{k \choose i+m} + 2^{-2(k-i)} {k \choose i}^2 \delta_i,
\end{equation}
where $\delta_i := 1$ \hspace{11pt} if $k-i$ is even\\
.\hspace{49pt} 0 \hspace{11pt} if $k-i$ is odd 

Similarly, $\cos^2(x/2)$ can appear in the expression of $D_i$ iff $w_{i-m,i+n}$ is a multiple of $\cos^2(x/2)$ and $m=n$, and $w_{i-m,i+m}$ is a multiple of $\cos^2(x/2)$ iff $k-i+m$ is odd. Therefore we have  
\begin{equation}\label{3.9}
    C_2^i = 2\sum_{\substack{ 1 \leq m \leq \text{min}\{i,k-i\} \\ k-i+m \hspace{2pt}\text{is odd}}} 2^{-2(k-i)} {k \choose i-m}{k \choose i+m} + 2^{-2(k-i)} {k \choose i}^2(1-\delta_i).
\end{equation}

Let us compute the coefficient of $\sin x$ in $D_i$. Observe that $\sin x$ can appear in $D_i$ in two different ways; first, when either $w_{i-m,i+n}$ is a multiple of $\sin^2(x/2)$ or $\cos^2(x/2)$ and $n-m$ is odd; secondly, when $w_{i-m,i+n}$ is a multiple of $\sin x$ and $n-m$ is even. Further, observing that $w_{i-m,i+n}$ is a multiple of $\sin^2(x/2)$ or $\cos^2(x/2)$ iff $n-m$ is even and $w_{i-m,i+n}$ is a multiple of $\sin x$ iff $n-m$ is odd implies that $C_4^i=0$. \\

After computing $C_1^i, C_2^i$ and $C_4^i$, it's not hard to see that
\begin{equation}\label{3.10}
    C_3^i = (-1)^{k-i-1}  2^{-2(k-i)}\sum_{\substack{0 \leq m \leq \text{min}\{i,k-i\} \\ m < n \leq k-i}} (-1)^n 2^{n-m} {k \choose i-m} {k \choose i+n}\Bigg({n-1 \choose m-1} + \frac{1}{2}{n-1 \choose m} \Bigg).
\end{equation}

Simplifying further, we find that $D_i(x) = (C_2^i+C_3^i) + (C_1^i - C_2^i -2C_3^i)\sin^2(x/2)$. Next we simplify the constants $(C_2^i+C_3^i)$ and $(C_1^i - C_2^i -2C_3^i)$. Let 
$$\xi_i^k:= \sum\limits_{\substack{0 \leq m \leq \text{min}\{i,k-i\} \\ 1 \leq n \leq k-i}} (-1)^n 2^{n-m} {k+1 \choose i-m} {k \choose i+n}{n-1 \choose m}$$ and consider
\begin{align*}
    (-1)^{k-i-1}C_3^i &= 2^{-2(k-i)} \sum_{\substack{0 \leq m \leq \text{min}\{i,k-i\} \\ m < n \leq k-i}} (-1)^n 2^{n-m} {k \choose i-m} {k \choose i+n}\Big({n-1 \choose m-1} + \frac{1}{2}{n-1 \choose m} \Big)\\
    &= \frac{2^{-2(k-i)}}{2}\xi_i^k - \sum_{1 \leq m \leq \text{min}\{i,k-i\}} (-1)^m 2^{-2(k-i)} {k \choose i-m}{k \choose i+m}\\
    &=\frac{2^{-2(k-i)}}{2}\xi_i^k - \frac{(-1)^{k-i}}{2} \Big(C_1^i-C_2^i + 2^{-2(k-i)}{k \choose i}^2(1-2\delta_i)\Big).
\end{align*}

Simplifying further we obtain 
\begin{equation}\label{3.11}
    2^{2(k-i)}\Big(C_1^{i}-C_2^{i}-2C_{3}^i\Big) = (-1)^{k-i}\xi_i^k + (-1)^{k-i} {k \choose i}^2.     
\end{equation}

Using the expression of $C_3^i$ from \eqref{3.11}, we get
\begin{equation}\label{3.12}
    \begin{split}
        2^{2(k-i)}\Big(C_2^i + C_3^i\Big) &= \sum_{1 \leq m \leq \text{min}\{i,k-i\}} {k \choose i-m} {k \choose i+m} + \frac{1}{2}{k \choose i}^2 - \frac{1}{2}(-1)^{k-i} {k \choose i}^2 - \frac{1}{2}(-1)^{k-i} \xi_i^k\\
        &= \frac{1}{2} {2k \choose 2i} - \frac{1}{2}(-1)^{k-i} {k \choose i}^2 - \frac{1}{2}(-1)^{k-i} \xi_i^k.
    \end{split}
\end{equation}

In the last step we used Chu-Vandermonde Identity: ${m+n \choose r } = \sum\limits_{i=0}^r {m \choose i}{n \choose r-i}$ with as change of variable.\\
\end{proof}

\begin{lem}\label{lem3.2}
Let $u$ be a function satisfying the hypothesis of Lemma \ref{lem3.1}. Furthermore, assume that $u$ has zero average, that is $\int_{-\pi}^\pi u dx = 0$. Then we have
\begin{equation}\label{3.13}
    \int_{-\pi}^\pi |u'|^2 \sin^2(x/2) dx \geq \frac{1}{16} \int_{-\pi}^\pi |u|^2 dx.
\end{equation}
\end{lem}
\begin{rem}
Inequality \eqref{3.13} is an improvement of well known \emph{Poincar\'e-Friedrichs inequality} in dimension one \cite[Theorem 258]{HLP1952}:
\begin{align*}
    \int_{-\pi}^\pi |u'(x)|^2 dx \geq \int_{-\pi}^\pi |u(x)|^2 dx,
\end{align*}
since $\sin^2(x/2) \leq 1$.
\end{rem}

\begin{proof}
Let $ w(x) := \frac{1}{4}\sec(x/2)$. Expanding the square we obtain
\begin{align*}
    |u' \sin(x/2) + w(u-u(\pi))|^2 &= |u'|^2 \sin^2(x/2) + w^2|u-u(\pi)|^2 + 2\text{Re}  [(w\sin(x/2))u'(\overline{u-u(\pi)})]\\
    &=|u'|^2 \sin^2(x/2) + w^2|u-u(\pi)|^2 + w\sin(x/2)(|u-u(\pi)|^2)'.
\end{align*}

Fix $\epsilon >0$. Doing integration by parts, we obtain
\begin{equation}\label{3.14}
    \int_{-\pi + \epsilon}^{\pi - \epsilon}  |u' \sin(x/2) + w(u-u(\pi))|^2 dx = \int_{-\pi + \epsilon}^{\pi -\epsilon} |u'|^2 \sin^2(x/2) + \int_{-\pi + \epsilon}^{\pi -\epsilon}(w^2 - (w\sin(x/2))')|u-u(\pi)|^2 dx + B.T. \geq 0,
\end{equation}
where the boundary term B.T. is given by
\begin{equation}\label{3.15}
    B.T.:= w(\pi - \epsilon)\sin((\pi-\epsilon)/2)|u(\pi-\epsilon)-u(\pi)|^2 - w(-\pi + \epsilon)\sin((-\pi+\epsilon)/2)|u(-\pi+\epsilon)-u(\pi)|^2.
\end{equation}

Therefore we have
\begin{equation}\label{3.16}
    \int_{-\pi + \epsilon}^{\pi -\epsilon} |u'|^2 \sin^2(x/2) dx \geq \int_{-\pi + \epsilon}^{\pi -\epsilon}(-w^2 + (w\sin(x/2))')|u-u(\pi)|^2 dx - B.T.
\end{equation}

Using $-w^2 + (w\sin(x/2))' = \frac{1}{16} \sec^2(x/2) \geq 1/16$ above, we obtain
\begin{equation}\label{3.17}
    \int_{-\pi + \epsilon}^{\pi -\epsilon} |u'|^2 \sin^2(x/2) dx \geq \frac{1}{16}\int_{-\pi + \epsilon}^{\pi -\epsilon}|u-u(\pi)|^2 dx - B.T.
\end{equation}

Using periodicity of $u$ along with the first order taylor expansion of $u$ around $\pi$ and $-\pi$, one can easily conclude that B.T. goes to 0 as $\epsilon$ goes to 0. Now taking limit $\epsilon \rightarrow 0$ on both sides of \eqref{3.17} and using dominated convergence theorem, we obtain
\begin{align*}
    \int_{-\pi}^\pi |u'|^2 \sin^2(x/2) dx &\geq \frac{1}{16}\int_{-\pi}^\pi |u-u(\pi)|^2 dx\\
    &= \frac{1}{16}\int_{-\pi}^\pi |u|^2 + \frac{1}{16}\int_{-\pi}^\pi |u(\pi)|^2 -\frac{2}{16}\text{Re}\overline{u(\pi)}\int_{-\pi}^\pi u  dx\\
    & \geq \frac{1}{16} \int_{-\pi}^\pi |u|^2 dx.
\end{align*}
\end{proof}

\begin{lem}\label{lem3.4}
Let $u$ be a function satisfying the hypotheses of Lemma \ref{lem3.2}. For $k \in \mathbb{N}$, the following holds
\begin{equation}\label{3.18}
    \int_{-\pi}^{\pi} |d^k(u(x)\sin(x/2))|^2 dx \geq \sum_{i= 0}^{k-1}\Big(\alpha_i^k + \frac{1}{16}\beta_{i+1}^k \Big) \int_{-\pi}^{\pi} |d^i u(x)|^2 dx + \beta_0^k \int_{-\pi}^\pi |u|^2 \sin^2(x/2) dx,
\end{equation}
where $\alpha_i^k$ and $\beta_i^k$ are as defined in \eqref{3.2} and \eqref{3.3} respectively. 
\end{lem}

\begin{proof}    
Let $f = d^{i-1} u$. Applying Lemma \ref{lem3.2} to $f$ we get
\begin{equation}\label{3.19}
    \int_{-\pi}^\pi |d^i u|^2 \sin^2(x/2)dx \geq \frac{1}{16} \int_{-\pi}^\pi |d^{i-1}u|^2 dx.
\end{equation}
Using \eqref{3.19} in \eqref{3.1} and using $\alpha_k^k =0$ gives the desired estimate \eqref{3.18}. Note that in proving \eqref{3.18} we have assumed the non-negativity of the constants $\beta_i^k$, which will be proved in Section \ref{sec6}.
\end{proof}

The next two lemmas are weighted versions of Lemmas \ref{lem3.2} and \ref{lem3.4} and will be used in proving the higher order Hardy inequalities.\\

\begin{lem}\label{lem3.5}
Let $u$ be a function satisfying the hypotheses of Lemma \ref{lem3.1}. Furthermore, assume that $\int_{-\pi}^\pi u \sin^{2k-2}(x/2) dx = 0$. For $k \geq 1$, we have
\begin{equation}\label{3.20}
    \int_{-\pi}^\pi |u'|^2 \sin^{2k}(x/2) dx \geq \frac{(4k-3)}{16} \int_{-\pi}^\pi |u|^2 \sin^{2k-2}(x/2) dx. 
\end{equation}

\end{lem}

\begin{proof}
Let $w := \frac{1}{4}\sin^{k-1}(x/2)\sec(x/2)$. Expanding the square, we obtain
\begin{align*}
    |u' \sin^k(x/2) + w(u-u(\pi))|^2 &= |u'|^2 \sin^{2k}(x/2) + w^2|u-u(\pi)|^2 + 2\sin^k(x/2) w \text{Re}[\overline{u'}(u-u(\pi))]\\
    &= |u'|^2 \sin^{2k}(x/2) + w^2|u-u(\pi)|^2 + \sin^k(x/2)w\Big(|u-u(\pi)|^2\Big)'.
\end{align*}

Now integrating over $(-\pi + \epsilon, \pi -\epsilon)$ for a fixed $\epsilon >0$, we get
\begin{align*}
    \int_{-\pi + \epsilon}^{\pi - \epsilon}|u' \sin^2(x/2) + w(u - u(\pi))|^2 & = \int_{-\pi+\epsilon}^{\pi-\epsilon} |u'|^2 \sin^{2k}(x/2) + \int_{-\pi+\epsilon}^{\pi-\epsilon} w^2|u-u(\pi)|^2\\
    &+  \int_{-\pi+\epsilon}^{\pi-\epsilon} w\sin^k(x/2)\Big(|u-u(\pi)|^2\Big)'. 
\end{align*}

Finally, using integrating by parts, we obtain
\begin{equation}\label{3.21}
    \begin{split}
        \int_{-\pi + \epsilon}^{\pi - \epsilon}|u' \sin^{2k}(x/2) + w(u - u(\pi))|^2 & = \int_{-\pi+\epsilon}^{\pi-\epsilon} |u'|^2 \sin^{2k}(x/2) + \int_{-\pi+\epsilon}^{\pi-\epsilon} w^2|u-u(\pi)|^2\\
        &- \int_{-\pi+\epsilon}^{\pi-\epsilon} \Big(w\sin^k(x/2)\Big)'|u-u(\pi)|^2 + B.T. \geq 0,
    \end{split}
\end{equation}
where the boundary term B.T. is given by
\begin{equation}\label{3.22}
    B.T. := |u(\pi-\epsilon)- u(\pi)|^2 w(\pi - \epsilon) \sin^k((\pi -\epsilon)/2) - |u(-\pi+\epsilon)- u(\pi)|^2 w(-\pi + \epsilon) \sin^k((-\pi + \epsilon)/2).  
\end{equation}

Now using $(w\sin^{k}(x/2))'- w^2 = \frac{1}{16} \sin^{2k-2}(x/2) \big(\sec^2(x/2) + 4k-4\big) \geq \frac{4k-3}{16} \sin^{2k-2}(x/2)$, we arrive at
\begin{equation}\label{3.23}
    \int_{-\pi+\epsilon}^{\pi-\epsilon} |u'|^2 \sin^{2k}(x/2) \geq \frac{(4k-3)}{16} \int_{-\pi+\epsilon}^{\pi-\epsilon} |u-u(\pi)|^2 \sin^{2k-2}(x/2) - B.T.
\end{equation}

Now taking limit $\epsilon \rightarrow 0$ on both sides of \eqref{3.23} and using dominated convergence theorem, we obtain
\begin{align*}
    \int_{-\pi}^\pi |u'|^2 \sin^{2k}(x/2) &\geq \frac{(4k-3)}{16} \int_{-\pi}^\pi |u-u(\pi)|^2 \sin^{2k-2}(x/2)\\
    &= \frac{(4k-3)}{16} \int_{-\pi}^\pi |u|^2 \sin^{2k-2}(x/2) + (4k-3)/16 \int_{-\pi}^\pi |u(\pi)|^2 \sin^{2k-2}(x/2)\\ &-\frac{(4k-3)}{8}\text{Re}\overline{u(\pi)} \int_{-\pi}^\pi u \sin^{2k-2}(x/2) \geq \frac{(4k-3)}{16} \int_{-\pi}^\pi |u|^2 \sin^{2k-2}(x/2). 
\end{align*}
\end{proof}

\begin{lem}\label{lem3.6}
Suppose $u$ satisfies the hypotheses of Lemma \ref{lem3.2}. Further, assume that $u$ has zero average. For $k \geq 2$, we have
\begin{equation}\label{3.24}
    \int_{-\pi}^\pi |d^k(u\sin^2(x/2))|^2 dx \geq \alpha_{k-1}^k \Big(\alpha_{k-2}^{k-1} + \frac{1}{16}\beta_{k-1}^{k-1}\Big) \int_{-\pi}^\pi |d^{k-2}u|^2 dx.
\end{equation}

\end{lem}

\begin{proof}
We begin with the observation that although $f=u \sin(x/2)$ does not satisfy the hypothesis $d^k f(-\pi) = d^k f(\pi)$ of Lemma \ref{lem3.1}, identity \eqref{3.1} still holds for $f$. In the proof of Lemma \ref{lem3.1}, the periodicity of derivatives is only used in the derivation of \eqref{3.6}; to make sure that no boundary term appears while doing integration by parts. The key observation is that $d^i f(-\pi) = -d^i f (\pi)$, which imply that $d^i f(-\pi) \overline{d^j f (-\pi)} = d^i f(\pi) \overline{d^j f (\pi)}$. This makes sure no boundary terms appears while performing integration by parts in \eqref{3.6} for the function $f$. \\

First using identity \eqref{3.1} for $u\sin(x/2)$ and then for $u$, along with non-negativity of the constants $\alpha_i^k$ and $\beta_i^k$ (will be proved in Section \ref{sec6}) we obtain
\begin{align*}
    \int_{-\pi}^\pi |d^k(u\sin^2(x/2))|^2dx &\geq \alpha_{k-1}^k \int_{-\pi}^\pi |d^{k-1}(u\sin(x/2))|^2 dx\\
    & \geq \alpha_{k-1}^k \Big(\alpha_{k-2}^{k-1} + \frac{1}{16}\beta_{k-1}^{k-1}\Big) \int_{-\pi}^\pi |d^{k-2}u|^2 dx.
\end{align*}
Last inequality uses Lemma \ref{lem3.2}.
\end{proof}

\section{Proof of Hardy inequalities}\label{sec4}

\begin{proof}[Proof of Theorem \ref{thm2.1}]
Let $u \in \ell^2(\mathbb{Z})$, we define its \textbf{Fourier transform} $\mathcal F(u)\in L^2((-\pi, \pi))$ as follows:
\begin{equation}\label{4.1}
    \mathcal{F}(u)(x) := (2\pi)^{-\frac{1}{2}}\sum_{n \in \mathbb{Z}} u(n) e^{-inx} \hspace{19pt} x \in (-\pi, \pi).
\end{equation}
Let $1 \leq j \leq k$. Using the inversion formula for Fourier transform and integration by parts, we get
\begin{align*}
    u(n)n^{k-j} &= (2\pi)^{-\frac{1}{2}} \int_{-\pi}^\pi \mathcal{F}(u)(x)n^{k-j} e^{inx}dx\\ 
    &=  \frac{(2\pi)^{-\frac{1}{2}}}{i^{k-j}} \int_{-\pi}^\pi \mathcal{F}(u)(x)d^{{k-j}} e^{inx}dx = \frac{(-1)^{k-j}(2\pi)^{-\frac{1}{2}}}{i^{k-j}}\int_{-\pi}^\pi d^{{k-j}} \mathcal{F}(u)(x) e^{inx}dx. 
\end{align*}
Applying Parseval's Identity gives us 
\begin{equation}\label{4.2}
    \sum_{n \in \mathbb{Z}}|u(n)|^2 n^{2(k-j)} = \int_{-\pi}^{\pi}|d^{k-j} \mathcal{F}(u)(x)|^2 dx.
\end{equation}
Similarly one gets the following identity 
\begin{equation}\label{4.3}
    \sum_{n \in \mathbb{Z}}|u(n)-u(n-1)|^2 \Big(n-\frac{1}{2}\Big)^{2k} = 4 \int_{-\pi}^\pi |d^k (\mathcal{F}(u)\sin(x/2))|^2 dx.
\end{equation}
Finally, applying  Lemma \ref{lem3.4} on $F(u)$ and then using \eqref{4.2}, \eqref{4.3} we get
\begin{equation}\label{4.4}
    \begin{split}
        \sum_{n \in \mathbb{Z}} |u(n)-u(n-1)|^2 \Big(n-\frac{1}{2}\Big)^{2k} &\geq \sum_{i=1}^k \gamma_i^k \sum_{n \in \mathbb{Z}} |u(n)|^2 n^{2(k-i)} + \beta_0^k \sum_{n \in \mathbb{Z}}|u(n)-u(n-1)|^2\\
        & \geq \sum_{i=1}^k \gamma_i^k \sum_{n \in \mathbb{Z}} |u(n)|^2 n^{2(k-i)} + \frac{\beta_0^k}{4} \sum_{n \in \mathbb{Z}\setminus\{0\}} \frac{|u(n)|^2}{n^2},
    \end{split}
\end{equation}
where $\gamma_i^k := 4\alpha_{k-i}^k + \frac{1}{4}\beta_{k-i+1}^k$.    In the last step we used the classical Hardy inequality. In Section \ref{sec6} we simplify the expressions of $\alpha_i^k$ and $\beta_i^k$, which will complete the proof of Theorem \ref{thm2.1}.
\end{proof}

\begin{proof}[Proof of Corollary \ref{cor2.2}]
Assuming $\gamma_i^k \geq 0$ (which will be proved in Section \ref{sec6}) Theorem \ref{thm2.1} immediately implies  
\begin{align*}
    \sum_{n \in \mathbb{Z}}|u(n)-u(n-1)|^2 \Big(n-\frac{1}{2}\Big)^{2k} \geq \gamma_1^k \sum_{n \in \mathbb{Z}} |u(n)|^2 n^{2k-2}.
\end{align*}
It can be checked that $\xi_{k-1}^k = -k(k+1)$. Using this in the expression of $\gamma_1^k$, we find that $\gamma_1^k = \frac{(2k-1)^2}{4}$. Next, we prove the sharpness of the constant $\gamma_1^k$. Let $C$ be a constant such that
\begin{equation}\label{4.5}
    \sum_{n \in \mathbb{Z}}|u(n)-u(n-1)|^2 \Big(n-\frac{1}{2}\Big)^{2k} \geq C \sum_{n \in \mathbb{Z}} |u(n)|^2 n^{2k-2}
\end{equation}
for all $u \in C_c(\mathbb{Z})$.\\
    
Let $N \in \mathbb{N}$, $\beta \in \mathbb{R}$ and $\alpha\geq 0$ be such that $2\beta + 2k-2 <-1$. Consider the following family of finitely supported functions on $\mathbb{Z}$. 
\begin{align*}
    u_{\beta,N}(n):=
    \begin{cases}
        n^\beta \hspace{73pt} &\text{for} \hspace{5pt}  1 \leq  n \leq N \\
        -N^{\beta-1} n + 2N^{\beta} \hspace{13pt}  &\text{for} \hspace{5pt}  N \leq n \leq 2N \\
        0 \hspace{83pt} &\text{for} \hspace{5pt}  n \geq 2N \hspace{5pt} \text{and} \hspace{5pt} n\leq 0
    \end{cases} 
\end{align*}
Clearly we have
\begin{equation}\label{4.6}
    \sum_{n \in \mathbb{Z}}|u_{\beta,N}(n)|^2n^{2k-2} \geq \sum_{n=1}^{N}n^{2\beta+2k-2}.
\end{equation}
and
\begin{equation}\label{4.7}
    \begin{split}
        \sum_{n \in \mathbb{Z}}|u_{\beta, N}(n)-u_{\beta,N}(n-1)|^2(n-1/2)^{2k}
        &=\sum_{n=1}^\infty|u_{\beta, N}(n)-u_{\beta,N}(n-1)|^2(n-1/2)^{2k} \\
        &\leq  \sum_{n=2}^{N}(n^\beta - (n-1)^\beta)^2n^{2k} + \sum_{n=N+1}^{2N}N^{2\beta-2}n^{2k} + 1.
    \end{split}
\end{equation}
Some basic estimates:
\begin{align*}
    (n^\beta - (n-1)^\beta)^2 &\leq \beta^2(n-1)^{2\beta-2},\\
    \sum_{n=N+1}^{2N} n^{2k} &\leq \int_{N+1}^{2N+1} x^{2k} dx = \frac{(2N+1)^{2k+1} - (N+1)^{2k+1}}{2k+1}.
\end{align*}
Using the above in \eqref{4.7}, we get
\begin{equation}\label{4.8}
    \begin{split}
        \sum_{n \in \mathbb{Z}}|u_{\beta, N}(n)-u_{\beta,N}(n-1)|^2(n-1/2)^{2k} &\leq \beta^2 \sum_{n=2}^{N}(n-1)^{2\beta-2}n^{2k} \\
        &+ \frac{N^{2\beta+2k-1}}{2k+1}\Bigg[\Big(2+\frac{1}{N}\Big)^{2k+1} - \Big(1+\frac{1}{N}\Big)^{2k+1}\Bigg]+1.
    \end{split}
\end{equation}
Using estimates \eqref{4.6} and \eqref{4.8} in \eqref{4.5} and taking limit $N \rightarrow \infty$, we get
\begin{align*}
    C\sum_{n=1}^{\infty}n^{2\beta+2k-2} &\leq \beta^2\sum_{n=2}^{\infty}(n-1)^{2\beta-2}n^{2k} + 1\\
    &= \beta^2 \sum_{i=0}^{2k}{2k \choose i}\sum_{n=1}^\infty n^{2\beta + 2k - i -2} + 1.
\end{align*}
Finally, taking limit $\beta \rightarrow \frac{1-2k}{2}$ on the both sides, we obtain
\begin{equation}\label{4.9}
    C \leq \frac{(2k-1)^2}{4}. 
\end{equation}
This proves the sharpness of $\gamma_1^k$.
\end{proof}
\section{Proof of Higher Order Hardy inequalities}\label{sec5}

\begin{proof}[Proof of Theorem \ref{thm2.3}]
First we prove inequality \eqref{2.4} and then inequality \eqref{2.5}.\\
Let $m \in \mathbb{N}$, $v \in C_c(\mathbb{Z})$ with $v(0)=0$ and
\begin{align*}
    \Tilde{v}(n):=
    \begin{cases}
        \frac{v(n)}{n^{2m}} \hspace{19pt} \text{if} \hspace{5pt} n \neq 0\\
        0 \hspace{33pt} \text{if} \hspace{5pt} n=0
    \end{cases}
\end{align*}
Using the inversion formula for Fourier transform, we obtain
\begin{align*}
    \Delta v = 2v(n)-v(n-1)-v(n+1) &=
    (2\pi)^{-\frac{1}{2}}\int_{-\pi}^\pi \mathcal{F}(v)(2-e^{-ix}-e^{ix})e^{inx} dx\\
    &= (2\pi)^{-\frac{1}{2}}\int_{-\pi}^\pi 4 \mathcal{F}(v)\sin^2(x/2)e^{inx} dx.
\end{align*}

Therefore we have $\mathcal{F}(\Delta v) = 4 \sin^2(x/2) \mathcal{F}(v)$. \hspace{-19pt} Applying this formula iteratively, we obtain $\mathcal{F}(\Delta^m v) = 4^m \sin^{2m}(x/2)\mathcal{F}(v)$. Using Parseval's, identity we get
\begin{align*}
    \sum_{n \in \mathbb{Z}\setminus\{0\}} \frac{|v|^2}{n^{4m}} &= \int_{-\pi}^\pi |\mathcal{F}(\Tilde{v})|^2 dx.\\
    \sum_{n \in \mathbb{Z}}|\Delta^m v|^2 &= 4^{2m}\int_{-\pi}^\pi |\mathcal{F}(v)|^2 \sin^{4m}(x/2) dx = 4^{2m}\int_{-\pi}^\pi |\mathcal{F}(\Tilde{v})^{(2m)}|^2 \sin^{4m}(x/2) dx.
\end{align*}
Using Lemma \ref{lem3.5} iteratively we obtain
\begin{align*}
    4^{-2m}\sum_{n \in \mathbb{Z}}|\Delta^m v|^2 =  \int_{-\pi}^\pi |\mathcal{F}(\Tilde{v})^{(2m)}|^2 \sin^{4m}(x/2) dx &\geq  \frac{1}{2^{8m}} \prod_{i=0}^{2m-1} (8m-3-4i)\int_{-\pi}^\pi |\mathcal{F}(\Tilde{v})|^2 dx\\
    &=\frac{1}{2^{8m}} \prod_{i=0}^{2m-1} (8m-3-4i) \sum_{n \in \mathbb{Z}\setminus\{0\}} \frac{|v|^2}{n^{4m}},
\end{align*}
under the assumption that 
\begin{equation}\label{5.1}
    \int_{-\pi}^\pi \mathcal{F}(\Tilde{v})^{(2m-k)} \sin^{2(2m-k)}(x/2)dx = \sum_{n\in \mathbb{Z}} \mathcal{F}^{-1}(\mathcal{F}(\Tilde{v})^{(2m-k)})(n)\mathcal{F}^{-1}(\sin^{2(2m-k)}(x/2)) = 0
\end{equation}
for $1 \leq k \leq 2m$. \\

Next we compute the inverse Fourier transform of $\sin^{2(2m-k)}(x/2)$ to simplify the condition \eqref{5.1}.
Consider
\begin{align*}
    \sin^{2(2m-k)}(x/2) &= 2^{-(2m-k)}(1-\cos x)^{2m-k}\\
    &= 2^{-(2m-k)} \sum_{j=0}^{2m-k}{2m-k \choose j} (-1/2)^j (e^{ix} + e^{-ix})^j \\
    &=2^{-(2m-k)} \sum_{j=0}^{2m-k}\sum_{j'=0}^j {2m-k \choose j} {j \choose j'}(-1/2)^j e^{-ix(2j'-j)}.
\end{align*}
Using the above expression, condition \eqref{5.1} becomes
\begin{align*}
    \sum_{n \in \mathbb{Z}}\mathcal{F}^{-1}(\mathcal{F}(\Tilde{v})^{(2m-k)})(n)&\mathcal{F}^{-1}(\sin^{2(2m-k)}(x/2)) \\
    &=  \sum_{j=0}^{2m-k}\sum_{\substack{0 \leq j' \leq j\\ j'\neq j/2 }} {2m-k \choose j} {j \choose j'}(-1/2)^j(2j'-j)^{-k} v(2j'-j) = 0.
\end{align*}
So finally we arrive at the following inequality
\begin{equation}\label{5.2}
    \sum_{n \in \mathbb{Z}}|\Delta^m v|^2 \geq \frac{1}{2^{4m}} \prod_{i=0}^{2m-1} (8m-3-4i) \sum_{n \in \mathbb{Z}\setminus\{0\}}\frac{|v|^2}{n^{4m}},
\end{equation}
provided $v \in C_c(\mathbb{Z})$ with $v(0) =0 $ satisfies
\begin{equation}\label{5.3}
    \sum_{j=0}^{2m-k}\sum_{\substack{0 \leq j' \leq j\\ j'\neq j/2 }} {2m-k \choose j} {j \choose j'}(-1/2)^j(2j'-j)^{-k} v(2j'-j) = 0,
\end{equation}
for $1 \leq k \leq 2m$.\\

Let $u \in C_c(\mathbb{N}_0)$ with $u(i)=0$ for all $0\leq i \leq 2m-1$. We define $v \in C_c(\mathbb{Z})$ as
\begin{align*}
    v(n):=
    \begin{cases}
        u(n) \hspace{19pt} \text{if} \hspace{5pt} n \geq 0\\
        0 \hspace{36pt} \text{if} \hspace{5pt} n <0
    \end{cases}
\end{align*}
It is quite straightforward to check that the condition \eqref{5.3} is trivially satisfied. Now applying inequality \eqref{5.2} to the above defined function $v$, we obtain
\begin{equation}\label{5.4}
    \sum_{n=1}^\infty |\Delta^m u|^2 \geq \frac{1}{2^{4m}} \prod_{i=0}^{2m-1} (8m-3-4i) \sum_{n=1}^\infty \frac{|u|^2}{n^{4m}}.
\end{equation}
This proves the inequality \eqref{2.4}. Inequality \eqref{2.5} can be proved in a similar way, by following the proof of \eqref{2.4} step by step. 
\end{proof}

\begin{proof}[Proof of Theorem \ref{thm2.6}]
First we prove inequality \eqref{2.7}. We begin by proving the result for $m=1$ and then apply the result for $m=1$ iteratively to prove it for general $m$. Using inversion formula and integration by parts, we obtain
\begin{align*}
    u(n)n^{k-2} &= (2\pi)^{-\frac{1}{2}} \int_{-\pi}^\pi \mathcal{F}(u)(x)n^{k-2} e^{inx}dx\\ 
    &=  \frac{(2\pi)^{-\frac{1}{2}}}{i^{k-2}} \int_{-\pi}^\pi \mathcal{F}(u)(x)d^{{k-2}} e^{inx}dx = \frac{(-1)^{k-2}(2\pi)^{-\frac{1}{2}}}{i^{k-2}}\int_{-\pi}^\pi d^{{k-2}} \mathcal{F}(u)(x) e^{inx}dx. 
\end{align*}
Applying Parseval's identity gives us 
\begin{equation}\label{5.5}
    \sum_{n \in \mathbb{Z}}|u(n)|^2 n^{2k-4} = \int_{-\pi}^{\pi}|d^{k-2} \mathcal{F}(u)(x)|^2 dx. 
\end{equation}
Similarly, one gets the following identity 
\begin{equation}\label{5.6}
    \sum_{n \in \mathbb{Z}}|\Delta u|^2 n^{2k} = 16 \int_{-\pi}^\pi |d^k (\mathcal{F}(u)\sin^2(x/2))|^2 dx.
\end{equation}
Now applying Lemma \ref{lem3.6} and then using equations \eqref{5.5} and \eqref{5.6}, we get
\begin{equation}\label{5.7}
    \begin{split}
        \sum_{n \in \mathbb{Z}}|\Delta u|^2 n^{2k} &\geq 16 \alpha_{k-1}^k \Big(\alpha_{k-2}^{k-1} + \frac{1}{16}\beta_{k-1}^{k-1}\Big)\\
        & = k(k-1)(k-3/2)^2 \sum_{n \in \mathbb{Z}} |u|^2 n^{2k-4}. 
    \end{split}
\end{equation}
In the last line we used $\alpha_{k-1}^k = k(k-1)$ and $\beta_k^k = 1$ (see \eqref{3.2}- \eqref{3.4}).
Now applying the inequality \eqref{5.7} inductively completes the proof of inequality \eqref{2.7}. For the proof of inequality \eqref{2.8}, we first apply inequality \eqref{2.3} and then inequality \eqref{2.7}. 
\end{proof}

\section{Combinatorial identity}\label{sec6}
In this section, we prove a combinatorial identity using the Lemma \ref{lem3.1}. This develops a very nice connection between combinatorial identities and functional identities. \hspace{-5pt} We believe that the method we present here can be used to prove new combinatorial identities which might be of some value. 
\begin{thm}\label{thm6.1}
Let $k \in \mathbb{N}$ and $ 0 \leq i \leq k$. Then
\begin{equation}\label{6.1}
    \sum_{\substack{0 \leq m \leq \text{min}\{i,k-i\} \\ 1 \leq n \leq k-i}} (-1)^n 2^{n-m} {k+1 \choose i-m} {k \choose i+n}{n-1 \choose m} = (-1)^{k-i} {k \choose i} - {k \choose i}^2.    
\end{equation}
\end{thm}

\begin{proof}
Using $\sin^2(x/2) = (1-\cos x)/2$, identity \eqref{3.1} can be re-written as 
\begin{equation}\label{6.2}
    \begin{split}
        \sum_{i=0}^{k}(-1)^{k-i}2^{-2(k-i)}\Big(\xi_i^k + {k \choose i}^2\Big) \int_{-\pi}^{\pi}|d^i u|^2 \cos x dx - &\sum_{i=0}^k 2^{-2(k-i)}{2k \choose 2i} \int_{-\pi}^{\pi} |d^i u|^2 dx \\
        &= -2 \int_{-\pi}^{\pi} |d^k(u \sin(x/2))|^2 dx.     
    \end{split}
\end{equation}
Let $u = e^{in(x/2)}\sin(x/2)$. Then some straightforward calculations give us the following identities for $m \geq 0$
\begin{equation}\label{6.3}
    2^{2m}\int_{\pi}^{\pi}|d^m u|^2 dx = \frac{\pi}{2} \Big((n+1)^{2m} + (n-1)^{2m} \Big).  
\end{equation}

\begin{equation}\label{6.4}
    2^{2m} \int_{-\pi}^{\pi} |d^m u|^2 \cos x = \frac{-\pi}{2} (n^2 -1)^m.
\end{equation}

\begin{equation}\label{6.5}
    2^{2k}\int_{-\pi}^{\pi}|d^k(u \sin (x/2))|^2 = \frac{\pi}{8} \Big( (n+2)^{2k} + (n-2)^{2k} + 4n^{2k}\Big).
\end{equation}
Using equations \eqref{6.3} - \eqref{6.5} in \eqref{6.2}, we obtain
\begin{equation}\label{6.6}
    \begin{split}
        -\frac{1}{2} \sum_{i=0}^k (-1)^{k-i} \Big(\xi_i^k + {k \choose i}^2 \Big)(n^2-1)^i &= \frac{1}{2} \sum_{i=0}^k {2k \choose 2i} \Big((n+1)^{2i} + (n-1)^{2i}\Big)\\
        &- \frac{1}{4} \Big((n+2)^{2k} + (n-2)^{2k} + 4n^{2k}\Big) \\
        &=  -\frac{1}{2} n^{2k}.
    \end{split}
\end{equation}
The last step uses 
\begin{equation}\label{6.7}
    \sum_{i=0}^k {2k \choose 2i} (n+1)^{2i} = \frac{1}{2}\Big((n+2)^{2k} + n^{2k}\Big).
\end{equation}
and 
\begin{equation}\label{6.8}
    \sum_{i=0}^k {2k \choose 2i} (n-1)^{2i} = \frac{1}{2} \Big((n-2)^{2k} + n^{2k}\Big).
\end{equation}
Therefore, for $n \in \mathbb{N}$, we have
\begin{equation}\label{6.9}
    \sum_{i=0}^k (-1)^{k-i} \Big(\xi_i^k + {k \choose i}^2 \Big)(n^2-1)^i =  \sum_{i=0}^{k}{k \choose i} (n^2 -1)^i,
\end{equation}
which implies the identity \eqref{6.1}.
\end{proof}

\begin{rem}
Using identity \eqref{6.1}, expressions of $\alpha_i^k , \beta_i^k$ defined by \eqref{3.2}, \eqref{3.3} respectively become
\begin{align*}
    2^{2(k-i)}\alpha_i^k = \frac{1}{2}{2k \choose 2i} - \frac{1}{2}{k \choose i} \hspace{19pt} \text{and} \hspace{19pt} 2^{2(k-i)}\beta_i^k = {k \choose i},
\end{align*}
and $\gamma_i^k := 4\alpha_{k-i}^k + \frac{1}{4}\beta_{k-i+1}^k$ becomes
\begin{align*}
    2^{2i}\gamma_i^k &= 2 {2k \choose 2i} - 2 {k \choose i} + {k \choose i-1}.
\end{align*}

From the above expressions, it is quite straightforward that the above constants are non-negative, thus justifying the assumptions used in the proofs of Lemma \ref{lem3.4}, lemma \ref{lem3.6} and Corollary \ref{cor2.2}. Finally, the expression of $\gamma_i^k$ along with \eqref{4.4} completes the proof of Theorem \ref{thm2.1}.
\end{rem}

\appendix
\section{Appendix}\label{appendix:A}
We prove identity \eqref{3.6} with $w_{ij}$ replaced with an arbitrary smooth $2\pi$ periodic funcition.

\begin{lem}
Let $u, w \in C^\infty[-\pi, \pi]$ such that their derivatives satisfy $d^k u(-\pi) = d^k u(\pi)$ and $d^k w(-\pi) = d^k w(\pi)$, for all $k \in \N_0$. Then for non-negative integers $0 \leq i<j$ we have
\begin{equation}\label{A1}
    I(i, j, w) :=\text{Re} \int_{-\pi}^{\pi} d^i u(x) \overline{d^ju(x)} w(x) dx = \sum_{\sigma = i}^{\lfloor\frac{i+j}{2}\rfloor} \int_{-\pi}^{\pi} C_{\sigma, w}^{i,j}(x) |d^\sigma u|^2,
\end{equation}
where $C_{\sigma,w}^{i,j}$ is given by
\begin{align*}
    C_{\sigma, w}^{i,j}(x) = {j-\sigma -1 \choose \sigma -i-1}(-1)^{j-\sigma}d^{i+j-2\sigma}w(x) + \frac{1}{2}{j-\sigma-1 \choose \sigma-i}(-1)^{j-\sigma}d^{i+j-2\sigma}w(x).
\end{align*}
\end{lem}

\begin{proof}
We prove the result using induction on the parameter $k:= j-i$. Let us assume that \eqref{A1} is true for all $0 \leq i < j$ such that $3 \leq j-i \leq k $. Consider non-negative integers $i<j$ such that $j-i = k+1$. Then integration by parts yields
\begin{align*}
    I(i, j, w) &= \text{Re} \int_{-\pi}^{\pi} d^i u(x) \overline{d^ju(x)} w(x) dx\\
    &= -\text{Re} \int_{-\pi}^{\pi} d^i u(x) \overline{d^{j-1}u(x)} w'(x) dx - \text{Re} \int_{-\pi}^{\pi} d^{i+1} u(x) \overline{d^{j-1}u(x)} w(x) dx\\
    &= -I(i, j-1, w') - I(i+1, j-1, w).
\end{align*}
Further using induction hypothesis we get
\begin{equation}\label{A2}
    \begin{split}
        I(i, j, w) = &\sum_{\sigma=i+1}^{\lfloor (i+j-1)/2 \rfloor} \int_{-\pi}^\pi \big(-C_{\sigma, w'}^{i, j-1} - C_{\sigma, w}^{i+1,j-1}\big) |d^\sigma u|^2 dx - \int_{-\pi}^\pi C_{i, w'}^{i, j-1}|d^{i}u|^2 dx \\
        &- \delta(i+j)\int_{-\pi}^\pi C_{\lfloor(i+j)/2 \rfloor ,w}^{i+1, j-1}|d^{\lfloor(i+j)/2 \rfloor}u|^2 dx,   
    \end{split}
\end{equation}
where $\delta(\text{odd numbers}) := 0$ and $\delta(\text{even numbers}) := 1$. Using identity ${n \choose r} + {n \choose r-1} = {n+1 \choose r}$ we obtain
\begin{equation}\label{A3}
    -C_{\sigma, w'}^{i, j-1} - C_{\sigma, w}^{i+1,j-1} = C_{\sigma, w}^{i, j}.
\end{equation}
It can also be checked that $-C_{i, w'}^{i, j-1} = C_{i, w}^{i, j} = \frac{1}{2}d^{j-1} w$ as well as $-C_{\lfloor(i+j)/2 \rfloor ,w}^{i+1, j-1} = C_{\lfloor(i+j)/2 \rfloor ,w}^{i, j} = (-1)^{j-i} w$ (for even $i+j$). These observations along with \eqref{A3} and \eqref{A2} proves \eqref{A1} for $3 \leq j-i = k+1$.\\

The base cases $j-i \in \{1, 2, 3\}$ can be checked by hand (it's a consequence of iterative integration by parts).

\end{proof}

\textbf{Acknowledgements:} It is a pleasure to thank Professor Ari Laptev for his input and encouragement and also for his comments on early drafts of this paper. The author would also like to thank Ashvni Narayanan for proof reading the document. Finally, we thank an anonymous reviewer for their thorough reading and many helpful suggestions. The author is supported by President's Ph.D. Scholarship, Imperial College London.\\

\end{document}